\documentclass{amsart}
\usepackage{graphicx}
\newtheorem{thm}{Theorem}[section]
\newtheorem{cor}[thm]{Corollary}
\newtheorem{lem}[thm]{Lemma}
\newtheorem{prop}[thm]{Proposition}
\theoremstyle{definition}

\theoremstyle{remark}
\newtheorem{rem}[thm]{Remark}
\numberwithin{equation}{section}

\begin{document}

\title[]{Solutions and stability of  generalized Kannappan's and Van Vleck's functional equations}%
\author{Elqorachi Elhoucien and Redouani Ahmed}%
\address{University Ibn Zohr, Faculty of Sciences, Department of Mathematics, Agadir, Morocco}%
\email{Eqorachi@hotmail.com; Redouani-ahmed@yahoo.fr}%

\begin{abstract}
We study  the solutions of the integral Kannappan's and Van Vleck's
functional equations
$$\int_{S}f(xyt)d\mu(t)+\int_{S}f(x\sigma(y)t)d\mu(t) = 2f(x)f(y),
\;x,y\in S;$$
$$\int_{S}f( x\sigma(y)t)d\mu(t)-\int_{S}f(xyt)d\mu(t) = 2f(x)f(y),
\;x,y\in S,$$  where $S$ is a semigroup, $\sigma$ is an involutive
automorphism of $S$ and $\mu$ is a linear combination of Dirac
measures $(\delta_{z_{i}})_{i\in I}$, such that for all $i\in I$,
$z_{i}$ is contained in the center of $S$. We show that the
solutions of these equations  are closely related to the solutions
of the  d'Alembert's classic functional equation with an involutive
automorphism. Furthermore, we obtain the superstability theorems
that these functional equations are superstable in the general case,
where $\sigma$ is an involutive morphism.
\end{abstract}
\maketitle
\section{Introduction}
Throughout this paper  $S$ denotes a semigroup: A set equipped with
an associative operation. We write the operation multiplicatively. A
function $\chi$ : $S\longrightarrow \mathbb{C}$ is said to be
multiplicative if $\chi(xy)=\chi(x)\chi(y)$ for all $x,y\in S.$  Let
$\sigma$ : $S\longrightarrow G$  denotes an involutive morphism,
that is $\sigma$  is an involutive automorphism:
$\sigma(xy)=\sigma(x)\sigma(y)$ and $\sigma(\sigma(x))=x$ for all
$x,y\in S$)  or  $\sigma$ is an  involutive anti-automorphism:
$\sigma(xy)=\sigma(y)\sigma(x)$ and
$\sigma(\sigma(x))=x$ for all $x,y\in S.$\\
 Van Vleck \cite{V1,V2} studied the continuous solutions $f$
: $\mathbb{R} \longrightarrow \mathbb{R}$, $f\neq 0$ of the
following functional equation
\begin{equation}\label{eq1}
f(x - y + z_0)-f(x + y + z_0) = 2f(x)f(y),\;x,y \in
\mathbb{R},\end{equation} where $z_0>0$ is fixed. He showed that any
continuous solution with minimal period $4z_0$ has to be the sine
function $f(x)=\sin(\frac{\pi}{2z_0}x)=\cos(\frac{\pi}{2z_0}(x-z_{0})$, $x\in \mathbb{R}$. \\
Kannappan \cite{K} proved that any solution  $f$:
$\mathbb{R}\longrightarrow \mathbb{C}$ of the functional equation
\begin{equation}\label{eq2} f(x+y+z_0)+f(x-y+z_0) = 2f(x)f(y),\;x,y
\in \mathbb{R}\end{equation} is periodic, if $z_0\neq 0$.
Furthermore, the periodic solutions  has the form $f(x)=g(x-z_0)$
where $g$ is a periodic solution of d'Alembert functional equation
\begin{equation}\label{eq3}
g(x+y)+g(x-y) = 2g(x)g(y),\;x,y \in \mathbb{R}.\end{equation}\\
Stetk\ae r [33, Exercise 9.18] found the complex-valued solutions of
the functional equation
\begin{equation}\label{eq3} f(xy^{-1}z_0)-f(xyz_0) =
2f(x)f(y),\;x,y\in G\end{equation} on groups $G$, where $z_0$ is a
fixed element  in the center of  $G$.
\\Perkins and Sahoo \cite{P} replaced the group inversion by an
involutive anti-automorphism $\sigma$: $G\longrightarrow G$ and they
obtained the abelian, complex-valued solutions of the functional
equation
\begin{equation}\label{eq4} f(x\sigma(y)z_0)-f(xyz_0) = 2f(x)f(y),\;x,y
\in G.\end{equation}  Stetk\ae r \cite{St3} extends the results of
Perkins and Sahoo \cite{P} about equation (\ref{eq4}) to the more
general case where $G$ is a semigroup and the solutions are not
assumed to be abelian and $z_0$ is a fixed element in the center of
$G$. \\Recently, Bouikhalene and Elqorachi \cite{Bouikha}  obtained
the solutions of an extension of Van Vleck's functional equation
\begin{equation}\label{eq600} \chi(y)f(x\sigma(y)z_0)-f(xyz_0) =
2f(x)f(y),\;x,y \in S\end{equation} on semigroup $S$, and where
$\chi$ is a multiplicative function such that $\chi(x\sigma(x))=1$
for all $x\in S.$ \\
There has been quite a development of the theory of d'Alembert's
functional equation \begin{equation}\label{eq6000}
    g(xy)+g(x\sigma(y))=2g(x)g(y),\; x,y\in G,
\end{equation} during the last ten years on non
abelian groups. The non-zero solutions of d'Alembert's functional
equation (\ref{eq6000}) for general groups, even monoids are the
normalized traces of certain representations of the group $G$ on
$\mathbb{C}^{2}$ \cite{da1,da2}.\\Stetk\ae r \cite{stkan} expressed
the complex-valued solutions of Kannappan's functional equation
\begin{equation}\label{eq200000} f(xyz_0)+f(x\sigma(y)z_0) =
2f(x)f(y),\;x,y \in S\end{equation} on semigroups in terms of
solutions of d'Alembert's functional equation (\ref{eq6000}).
\\Elqorachi \cite{Elqorachi} extended the results of Stetk\ae r
\cite{stkan, St3} to the following generalizations of Kannappan's
functional equation
\begin{equation}\label{eq02000}
\int_{S}f(xyt)d\mu(t)+\int_{S}f(x\sigma(y)t)d\mu(t) = 2f(x)f(y),
\;x,y\in S
\end{equation} and Van Vleck functional equation
\begin{equation}\label{eq030000}\int_{S}f( x\sigma(y)t)d\mu(t)-\int_{S}f(xyt)d\mu(t) = 2f(x)f(y),
\;x,y\in S,\end{equation} where  $\mu$ is a linear combination of
Dirac measures $(\delta_{z_{i}})_{i\in I}$, with $z_{i}$  contained
in the center of the semigroup $S$, for all $i\in I$ and where
$\sigma$ is an involutive anti-automorphism of $S$.\\ Recently,
Zeglami and  Fadli \cite{zeglami} obtained the continuous and
central solutions of (\ref{eq02000}) and (\ref{eq030000}) on locally
compact groups. Related studies of functional equations
like (\ref{eq02000}) can be found in \cite{ba, akk2, el1, el2}.\\
The stability of functional equations  highlighted a phenomenon
which is usually called superstability. Consider the functional
equation $ E(f ) = 0$ and assume we are in a framework where the
notion of boundedness of $f$ and of $E( f)$ makes sense. We say that
the equation $ E(f ) = 0$ is superstable if the boundedness of $E(f
)$ implies that either $f$ is bounded or $f$ is a solution of $ E(f
) = 0$. This property was first observed when the following theorem
was proved by Baker, Lawrence, and  Zorzitto \cite{j1}: Let $V$ ba a
vector space. If a function $f$: $V\longrightarrow \mathbb{R}$
  satisfies the inequality $| f(x+y)-f(x)f(y)|\leq \varepsilon$ for some $\varepsilon>0$ and for all $x,y\in V$, then either $f$ is
bounded on $V$ or $f(x+y)=f(x)f(y)$ for all $x,y\in V.$
  \\
The result was generalized by  Baker \cite{j2}, by replacing $V$ by
a semigroup and $\mathbb{R}$ by a normed algebra $E$, in which the
norm is multiplicative, by  Ger and  \v{S}emrl \cite{se}, where $E$
is an arbitrary commutative complex semisimple Banach algebra and by
Lawrence \cite{la} in the case where $E$ is the algebra of all
$n\times n$ matrices. Different generalization of the result of
Baker, Lawrence and Zorzitto have been obtained. We mention for
example \cite{b}, \cite{e},  \cite{ger}, \cite{5}, \cite{10},
\cite{kim1}, \cite{kim2} and \cite{red}.
\\The first purpose of this paper is to extend the
results of Stetk\ae r \cite{St3, stkan} about the  Kannappan's
functional equation (\ref{eq02000}) and Van Vleck's functional
equation (\ref{eq030000}) to the case, where $\sigma$ is an
involutive automorphism of $S$.\\ By using similar methods and
computations to those in \cite{Elqorachi} we prove that the
solutions of (\ref{eq02000}) and (\ref{eq030000}) are also closely
related to the solutions of the   d'Alembert's classic functional
equation (\ref{eq6000}) (with $\sigma$ an involutive automorphism)
which has not been studied much on non-abelian semigroups.
Exceptions are Stetk\ae r [\cite{sph}, Example 6] (continuous
solutions), Sinopoulos \cite{sino} (general solutions) for a special
involutive automorphism $\sigma$ of the Heisenberg group. We show
that any solution of (\ref{eq6000}) is proportional to a solution of
(\ref{eq6000}). We prove that all solutions of the integral Van
Vleck's functional equation (\ref{eq030000}) are abelian and as
 an application we obtain some results  about abelian solutions of (\ref{eq6000}).\\
 In our proofs we do not need the crucial proposition [33,
Proposition 8.14] used in the proofs of the main results in \cite{Elqorachi} and \cite{St3, stkan}.  \\
The second purpose of this paper is to prove the superstability of
equations (\ref{eq02000}) and (\ref{eq030000}).  We show that the
superstability of these functional equations is closely related to
the superstability of the  Wilson's classic functional equation
$$f(xy)+f(x\sigma(y))=2f(x)g(y)\;x,y\in S,$$
 and consequently,  we obtain the superstability theorems of equations
(\ref{eq02000}) and (\ref{eq030000}) on semigroups that are not
necessarily abelian and where $\sigma$ is an involutive morphism.
\section{Integral Kannappan's functional equation on semigroups}
In this section we study the complex-valued solutions of the
functional equation (\ref{eq02000}), where $\sigma$ is an involutive
automorphism  and $\mu$ is a linear combination of Dirac measures
$(\delta_{z_{i}})_{i\in I}$, such that  $z_{i}$ is contained in the
center of $S$ for all $i\in I$.\\Throughout this paper we use in
(all) proofs without explicit mentioning the assumption that for all
$i\in I$; $z_i$ is in the center of $S$ and its consequence
$\sigma(z_i)$ is in the center of $S$. The following lemma has been
obtained in \cite{Elqorachi} for $\sigma$ an involutive
anti-automorphism. It is still true, where $\sigma$ is an involutive
automorphism. In the proof we adapt similar computations used in
\cite{Elqorachi}.
\begin{lem}  If $f$: $S\longrightarrow \mathbb{C}$ is a solution of (\ref{eq02000}), then  for
all $x\in S$
\begin{equation}\label{eqo77}
    f(x)=f(\sigma(x)),
\end{equation}
\begin{equation}\label{eqo88}
    \int_{S}f(t)d\mu(t)\neq 0\Longleftrightarrow f\neq 0,
\end{equation}
\begin{equation}\label{eqo999}
    \int_{S}\int_{S}f(x\sigma(t)s)d\mu(t)d\mu(s)=f(x)\int_{S}f(t)d\mu(t),
\end{equation}
\begin{equation}\label{eqo1000}
    \int_{S}\int_{S}f(xts)d\mu(t)d\mu(s)=f(x)\int_{S}f(t)d\mu(t).
\end{equation}
  \end{lem}
The following notations will be used later.\\
- $\mathcal{A}$ consists of the solutions of $g:$ $S\longrightarrow
\mathbb{C}$ of d'Alembert's functional equation (\ref{eq6000}) with
$\int_{S}g(t)d\mu(t)\neq 0$ and satisfying the condition
\begin{equation}\label{pr2}
  \int_{S}g(xt)d\mu(t)=g(x)\int_{S}g(t)d\mu(t)\; \text{for all}\;x\in S.\end{equation}
\\- To any $g\in \mathcal{A}$ we associate the function
$Tg=\int_{S}g(t)d\mu(t)g:$ $S\longrightarrow \mathbb{C}$.\\
- $\mathcal{K}$ consists of the non-zero solutions $f:$
$S\longrightarrow \mathbb{C}$ of the integral Kannappan's functional
equation (\ref{eq02000}). \\In the following theorem the complex
solutions of equation (\ref{eq02000}) are expressed by means of
solutions of d'Alembert's functional equation (\ref{eq6000}).
\begin{thm}{(1)} $T$ is a bijection of
$\mathcal{A}$ onto $\mathcal{K}$. The inverse $T^{-1}$:
$\mathcal{K}\longrightarrow \mathcal{A}$ is given by the formula
$$(T^{-1}f)(x)=\frac{\int_{S}f(xt)d\mu(t)}{\int_{S}f(t)d\mu(t)}$$ for
all $f\in \mathcal{K}$ and $x\in S. $\\{(2)} Any non-zero solution
$f$: $S\longrightarrow \mathbb{C}$ of the integral Kannappan's
functional equation (\ref{eq02000}) is of the form
$f=\int_{S}g(t)d\mu(t)g$, where $g\in \mathcal{A}$. Furthermore,
$f(x)=\int_{S}g(xt)d\mu(t)=\int_{S}g(x\sigma(t))d\mu(t)=\int_{S}g(t)d\mu(t)g(x)$
for all $x\in S.$\\(3) $f$ is central i.e. $f(xy)=f(yx)$ for all
$x,y\in S$  if and only if $g$ is central.\\(4) If $S$ is equipped
with a topology and $\sigma$: $S\longrightarrow S$ is continuous
then $f$ is continuous  if and only if $g$ is continuous.
\end{thm} \begin{proof} Similar computations to those of   \cite{Elqorachi}, where  $\sigma$ anti-automorphism involutive, can  be adapted  to the present situation.
The only assertion
 that need  proof is that the function
$$g(x)=\frac{\int_{S}f(xt)d\mu(t)}{\int_{S}f(t)d\mu(t)}$$   defined in \cite{Elqorachi} satisfies the
condition (\ref{pr2}).\\ By replacing $x$ by $xks$ and $y$ by $r$ in
(\ref{eq02000}) and integrating the result obtained with respect to
$k$, $s$ and $r$ we get
\begin{equation}\label{red1}
\int_{S}\int_{S}\int_{S}\int_{S}f(xksrt)d\mu(k)d\mu(s)d\mu(r)d\mu(t)+\int_{S}\int_{S}\int_{S}\int_{S}f(xks\sigma(r)t)d\mu(k)d\mu(s)d\mu(r)d\mu(t)\end{equation}
$$=2\int_{S}\int_{S}f(xks)d\mu(k)d\mu(s)\int_{S}f(r)d\mu(r)=2f(x)\bigg(\int_{S}f(s)d\mu(s)\bigg)^{2}.$$
By replacing $x$  by $xs$ and $y$ by $kr$ in (\ref{eq02000}) and
integrating the result obtained with respect to $k$, $s$ and $r$ we
obtain
\begin{equation}\label{red2}
\int_{S}\int_{S}\int_{S}\int_{S}f(xskrt)d\mu(s)d\mu(k)d\mu(r)d\mu(t)+\int_{S}\int_{S}\int_{S}\int_{S}f(xs\sigma(r)\sigma(k)t)d\mu(k)d\mu(s)d\mu(r)d\mu(t)\end{equation}
$$=2\int_{S}\int_{S}f(kr)d\mu(k)d\mu(r)\int_{S}f(xs)d\mu(s).$$
From (\ref{eqo999}) and (\ref{eqo1000}) we have
$$\int_{S}\int_{S}\int_{S}\int_{S}f(xks\sigma(r)t)d\mu(k)d\mu(s)d\mu(r)d\mu(t)=\int_{S}\int_{S}f(xks)d\mu(k)d\mu(s)\int_{S}f(s)d\mu(s)$$
$$=f(x)\bigg(\int_{S}f(s)d\mu(s)\bigg)^{2}$$
and
$$\int_{S}\int_{S}\int_{S}\int_{S}f(xr\sigma(k)t\sigma(s))d\mu(k)d\mu(s)d\mu(r)d\mu(t)=\int_{S}\int_{S}\int_{S}\int_{S}f(xr\sigma(k))d\mu(r)d\mu(k)\int_{S}f(s)d\mu(s)$$
$$=f(x)\bigg(\int_{S}f(s)d\mu(s)\bigg)^{2}.$$ In
view of (\ref{red1}) and (\ref{red2}) we deduce that
$$
\int_{S}\int_{S}f(kr)d\mu(k)d\mu(r)\int_{S}f(xs)d\mu(s)=\int_{S}\int_{S}f(xks)d\mu(k)d\mu(s)\int_{S}f(r)d\mu(r).$$
So, by using the expression of $g$ we obtain
$$\int_{S}g(xs)d\mu(s)=g(x)\int_{S}f(s)d\mu(s)$$ for all $x\in S.$
This completes the proof.
\end{proof}\begin{rem} In Stetk\ae r's paper \cite{stkan} about Kannappan's functional equation on semigroups,
more precisely in the definition of the set $\mathcal{A}$  other
assertions  which are equivalent to (\ref{pr2})  are needed to prove
the main result in \cite{stkan}. We notice here that we do not need
these statements to show the main result. The same note  is also
valid   for
 the manuscript \cite{Elqorachi}.
\end{rem}
Now, we extend Stetk\ae r's result \cite{stkan} from
anti-automorphisms to the more general case of morphism as follows.
\begin{cor}Let $z_0$ be a fixed element in the center of a semigroup $S$ and let  $\sigma$ be an involutive morphism of $S$. Then, any non-zero solution
$f$: $S\longrightarrow \mathbb{C}$ of  the functional equation
(\ref{eq200000}) is of the form $f=g(z_0)g$, where $g$ is a
 solution of d'Alembert's functional equation (\ref{eq6000})
with $g(z_0)\neq 0$ and satisfying the condition
$g(xz_0)=g(z_0)g(x)$ for all $x\in S$.
\end{cor}\begin{cor}If $\sigma=I$, where $I$ is the identity map of $S$. Then, any non-zero solution
$f$: $S\longrightarrow \mathbb{C}$ of  Kannappan's functional
equation $$\int_{S}f(xyt)d\mu(t)=f(x)f(y),\;x,y\in S$$ is of the
form $f=\chi\int_{S}\chi(t)d\mu(t)$, where $\chi$ is a
multiplicative function such that $\int_{S}\chi(t)d\mu(t)\neq
0$.\end{cor}
\begin{rem} The result of the Corollary 2.5 is also true without the assumption  that $\mu$  is a linear
combination of Dirac measures  $\delta_{z_{i}}$ with  $z_{i}$
 contained in the center of $S$ (see \cite{salma}).
 \end{rem}
\begin{cor} The non-zero central solutions of the integral Kannappan's
functional equation (\ref{eq02000}), where $\sigma$ is an involutive
automorphism of $S$ are the functions of the form
$$f(x)=\bigg[\frac{\chi(x)+\chi(\sigma(x))}{2}\bigg]\int_{S}\chi(t)d\mu(t),\; x\in
S,$$ where $\chi:$ $S\longrightarrow \mathbb{C}$ is a multiplicative
function such that\\ $\int_{S}\chi(t)d\mu(t)\neq 0$ and
$\int_{S}\chi(\sigma(t))d\mu(t)=\int_{S}\chi(t)d\mu(t)$.
\end{cor}
\begin{proof} From Theorem 2.2, if $f$ is a central solution of (\ref{eq02000}) then $g$ is a central solution
of d'Alembert's functional equation (1.9), with $\sigma$  an
involutive automorphism of $S$. In view of \cite{stetkaer}, there
exists a non-zero multiplicative function $\chi$: $S\longrightarrow
\mathbb{C}$ such that\begin{equation}\label{alembert}
g(x)=\frac{\chi(x)+\chi(\sigma(x))}{2}\end{equation}
 for all $x\in S.$ So, $f(x)=[\frac{\chi(x)+\chi(\sigma(x))}{2}]\int_{S}\chi(t)d\mu(t)$
with $\int_{S}\chi(t)d\mu(t)\neq0$. On the other hand by
substituting the condition
$\int_{S}g(xt)d\mu(t)=g(x)\int_{S}g(t)d\mu(t)$ into (\ref{alembert})
we get $\int_{S}\chi(\sigma(t))d\mu(t)=\int_{S}\chi(t)d\mu(t)$. This
completes the proof.
\end{proof}
\section{Superstability of the Intergral Kannappan functional
equation (\ref{eq02000})}In this section we obtain the
superstability result of equation (\ref{eq02000}) on semigroups not
necessarily abelian. Later, we need the following Lemma.
\begin{lem}  Let $\sigma$ be an involutive morphism  of $S$. Let $\mu$ be a complex
measure that is a linear combination of Dirac measures
$(\delta_{z_{i}})_{i\in I}$, such that  $z_{i}$ is contained in the
center of $S$ for all $i\in I$. Let $\delta>0$ be fixed.  If $f:
S\longrightarrow \mathbb{C}$ is an  unbounded function which
satisfies the inequality
\begin{equation}\label{ahm1}
 |\int_{S}f(xyt)d\mu(t)+\int_{S}f(x\sigma(y)t)d\mu(t)-2f(x)f(y)|\leq \delta
\end{equation} for all $x,y\in S$. Then, for all $x\in S$
\begin{equation}\label{ahm2}
    f(\sigma(x))=f(x),
\end{equation}
\begin{equation}\label{ahm3}
|\int_{S}\int_{S}f(x\sigma(s)t)d\mu(s)d\mu(t)-f(x)\int_{S}f(t)d\mu(t)|\leq
\frac{\delta}{2}\|\mu\|,
\end{equation}
\begin{equation}\label{ahm4}
|\int_{S}\int_{S}f(xst)d\mu(s)d\mu(t)-f(x)\int_{S}f(t)d\mu(t)|\leq
\frac{3\delta}{2}\|\mu\|,
\end{equation}
\begin{equation}\label{ahm4'}
\int_{S}f(t)d\mu(t)\neq 0.
\end{equation}
The function $g$ defined by
\begin{equation}\label{ahm5}g(x)=
\frac{\int_{S}f(xt)d\mu(t)}{\int_{S}f(t)d\mu(t)}\; \text{for}\;x\in
S\end{equation} is unbounded on $S$ and satisfies the following
inequalities.
\begin{equation}\label{ahm6}
 |g(xy)+g(x\sigma(y))-2g(x)g(y)|\leq
 \frac{3\delta}{(\int_{S}f(s)d\mu(s))^{2}}\|\mu\|^{2},
\end{equation}
\begin{equation}\label{ahm7}
 |\int_{S}g(xt)d\mu(t)-g(x)\int_{S}g(t)d\mu(t)|\leq
 \frac{(5/4)\delta\|\mu\|^{3}+(1/4)\delta\|\mu\|^{2}}{(|\int_{S}f(s)d\mu(s)|)^{2}}+\frac{\delta\|\mu\|}{|\int_{S}f(s)d\mu(s)|}
\end{equation} for all $x,y\in S$. Furthermore,
$g$ is a non-zero solution of  d'Alembert's functional equation
(\ref{eq6000}) and satisfies the condition (\ref{pr2}). That is
$T^{-1}f=g\in \mathcal{A}.$
\end{lem}
 \begin{proof} Equation (3.2): Replacing $y$ by $\sigma(y)$ in (\ref{ahm1}) and subtracting resulting inequalities we find after using the triangle inequality that
 $|f(x)(f(y)-f(\sigma(y)))|\leq
 2\delta$. Since $f$ is assumed to be unbounded then
 $f(\sigma(y))=f(y)$ for all $y\in S.$\\
Equation (3.2): By replacing $x$ by $\sigma(s)$ in (\ref{ahm1}) and
integrating the result obtained with respect  to $s$ we get
 $$|\int_{S}\int_{S}f(\sigma(s)yt)d\mu(t)d\mu(s)+\int_{S}\int_{S}f(\sigma(s)\sigma(y)t)d\mu(t)d\mu(s)-2f(y)\int_{S}
 f(\sigma(s))d\mu(s)|$$$$\leq\delta\|\mu\|,$$ which can be written
 $$|\int_{S}\int_{S}f(\sigma(s)yt)d\mu(t)d\mu(s)+\int_{S}\int_{S}f(\sigma(s)yt)d\mu(t)d\mu(s)-2f(y)\int_{S}
 f(s)d\mu(s)|\leq\delta\|\mu\|,$$
 because $f\circ\sigma=f$.  This proves (\ref{ahm3}).\\ Equation (3.4): By setting $y=s$ in
 (\ref{ahm1}) and integrating the result obtained with respect to $s$ we get
 $$|\int_{S}\int_{S}f(xst)d\mu(t)d\mu(s)+\int_{S}\int_{S}f(x\sigma(s)t)d\mu(t)d\mu(s)-2f(x)\int_{S}
f(s)d\mu(s)|\leq\delta\|\mu\|.$$ According to  (\ref{ahm3}) and the
triangle inequality we deduce (\ref{ahm4}).
 \\Equation (3.5): Assume that $f$ is an unbounded function which satisfies the
 inequality
 (\ref{ahm1}) and that $\int_{S}f(t)d\mu(t)=0$. Replacing $x$ by
 $xs$, $y$ by $yk$ in  (\ref{ahm1}) and integrating the result obtained with respect
 to $s$ and $k$ we get
\begin{equation}\label{ahm12}
|\int_{S}\int_{S}\int_{S}f(xsykt)d\mu(t)d\mu(s)d\mu(k)+\int_{S}\int_{S}\int_{S}f(xs\sigma(yk)t)d\mu(t)d\mu(s)d\mu(k)\end{equation}
$$-2\int_{S}f(xs)d\mu(s)\int_{S}f(yt)d\mu(s)|\leq\delta\|\mu\|^{2}.$$
In view of (\ref{ahm3}) and (\ref{ahm4}) we have
$$|\int_{S}\int_{S}\int_{S}f(xs\sigma(t)\sigma(y)k)d\mu(t)d\mu(s)d\mu(k)$$$$-\int_{S}[\int_{S}f(t)d\mu(t)f(xs\sigma(y))]d\mu(s)|\leq\frac{\delta}{2}\|\mu\|^{2},$$
$$|\int_{S}\int_{S}\int_{S}f(xsytk)d\mu(t)d\mu(s)d\mu(k)-
\int_{S}f(t)d\mu(t)\int_{S}f(xys)d\mu(s)|\leq\frac{3\delta}{2}\|\mu\|^{2}.$$
Since $\int_{S}f(t)d\mu(t)=0$, then we get
$$|\int_{S}\int_{S}\int_{S}f(xs\sigma(t)\sigma(y)k)d\mu(t)d\mu(s)d\mu(k)|\leq\frac{\delta}{2}\|\mu\|^{2},$$
$$|\int_{S}\int_{S}\int_{S}f(xsytk)d\mu(t)d\mu(s)d\mu(k)|\leq\frac{3\delta}{2}\|\mu\|^{2}.$$
From (\ref{ahm12}) we conclude that the function
$h(x)=\int_{S}f(xs)d\mu(s)$ is a bounded function  on $S$, in
particular the functions
$(x,y)\longrightarrow$$\int_{S}f(xys)d\mu(s)$;
$(x,y)\longrightarrow$$\int_{S}f(x\sigma(y)s)d\mu(s)$ are bounded on
$S\times S$. So, from (\ref{ahm1}) we deduce that $f$ is a bounded
function, which contradict the assumption that $f$ is an unbounded
function on $S$ and this proves (\ref{ahm4'}). \\
Equation (3.7): In the following we will show that the function $g$
defined by (\ref{ahm5}) is unbounded. If $g$ is bounded, then there
exists $M> 0$ such that $|\int_{S}f(xs)d\mu(s)|\leq M$ for all $x\in
S$. From (\ref{ahm1}) and the triangle inequality we get that the
function $(x,y)\longrightarrow f(x)f(y)$ is bounded on $S\times S$
and this implies that $f$ is bounded. This contradict the fact that
$f$ is assumed to be unbounded on $S$.\\ From the inequalities
(\ref{ahm1}), (\ref{ahm3}) and (\ref{ahm4}), we get
$$(\int_{S}f(s)d\mu(s))^{2}[g(xy)+g(x\sigma(y))-2g(x)g(y)]$$
$$=\int_{S}f(s)d\mu(s)\int_{S}f(xyt)d\mu(t)+\int_{S}f(s)d\mu(s)\int_{S}f(x\sigma(y)t)d\mu(t)$$
$$-2\int_{S}f(xk)d\mu(k)\int_{S}f(ys)d\mu(s)$$
$$=\int_{S}f(s)d\mu(s)\int_{S}f(xyt)d\mu(t)-\int_{S}\int_{S}\int_{S}f(xytks)d\mu(t)d\mu(s)d\mu(k)$$
$$+\int_{S}f(s)d\mu(s)\int_{S}f(x\sigma(y)t)d\mu(t)-\int_{S}\int_{S}\int_{S}f(x\sigma(y)t\sigma(s)k)d\mu(t)d\mu(s)d\mu(k)$$
$$+\int_{S}\int_{S}\int_{S}f(xk\sigma(ys)t)d\mu(t)d\mu(s)d\mu(k)+\int_{S}\int_{S}\int_{S}f(xkyst)d\mu(t)d\mu(s)d\mu(k)$$
$$-2\int_{S}f(xk)d\mu(k)\int_{S}f(ys)d\mu(s)$$
$$\leq\frac{3\delta}{2}\|\mu\|^{2}+\frac{\delta}{2}\|\mu\|^{2}+\delta\|\mu\|^{2}.$$
Which gives (\ref{ahm6}). \\Equation (3.8): For all $x\in S$, we
have
$$\int_{S}g(xs)d\mu(s)-g(x)\int_{S}g(t)d\mu(t)$$$$=\frac{\int_{S}\int_{S}f(xst)d\mu(s)d\mu(t)}{\int_{S}f(s)d\mu(s)}$$$$-\frac{\int_{S}\int_{S}f(ks)d\mu(k)d\mu(s)
\int_{S}f(xs)d\mu(s)}{(\int_{S}f(s)d\mu(s))^{2}}$$
$$=\frac{\int_{S}\int_{S}f(xst)d\mu(s)d\mu(t)\int_{S}f(s)d\mu(s)-\int_{S}\int_{S}f(ks)d\mu(k)d\mu(s)
\int_{S}f(xs)d\mu(s)}{(\int_{S}f(s)d\mu(s))^{2}}.$$ Replacing $x$ by
$xsk$ and $y$ by $r$ in (\ref{ahm1}) and integrating  the result
obtained with respect to $s$, $k$ and $r$ we get
\begin{equation}\label{ahm20}
 |\int_{S}\int_{S}\int_{S}\int_{S}f(xskrt)d\mu(s)d\mu(k)d\mu(r)d\mu(t)+\int_{S}\int_{S}\int_{S}\int_{S}f(xsk\sigma(r)t)d\mu(s)d\mu(k)d\mu(r)d\mu(t)-
 \end{equation}$$2\int_{S}\int_{S}f(xsk)d\mu(s)d\mu(k)\int_{S}f(r)d\mu(r)|\leq
 \delta \|\mu\|^{3}.$$
By replacing $x$ by $xs$ and $y$ by $kr$ in (\ref{ahm1}) and
integrating the result obtained with respect to $s$, $k$ and $r$ we
get \begin{equation}\label{ahm21}
 |\int_{S}\int_{S}\int_{S}\int_{S}f(xskrt)d\mu(s)d\mu(k)d\mu(r)d\mu(t)+\int_{S}\int_{S}\int_{S}\int_{S}f(xs\sigma(k)\sigma(r)t)d\mu(s)d\mu(k)d\mu(r)d\mu(t)-
 \end{equation}$$2\int_{S}\int_{S}f(kr)d\mu(k)d\mu(r)\int_{S}f(xs)d\mu(s)|\leq
 \delta \|\mu\|^{3}.
$$Since
$$2\int_{S}\int_{S}f(ks)d\mu(k)d\mu(s)
\int_{S}f(xs)d\mu(s)-2\int_{S}f(xst)d\mu(s)d\mu(t)\int_{S}f(s)d\mu(s)$$
$$=[2\int_{S}\int_{S}f(ks)d\mu(k)d\mu(s)
\int_{S}f(xs)d\mu(s)-\int_{S}\int_{S}\int_{S}\int_{S}f(xsrkt)d\mu(k)d\mu(s)d\mu(r)d\mu(t)$$
$$-\int_{S}\int_{S}\int_{S}\int_{S}f(xsr\sigma(k)\sigma(t))d\mu(k)d\mu(s)d\mu(r)d\mu(t)]$$
$$-[2\int_{S}f(xst)d\mu(s)d\mu(t)\int_{S}f(s)d\mu(s)-\int_{S}\int_{S}\int_{S}\int_{S}f(xsrkt)d\mu(k)d\mu(s)d\mu(r)d\mu(t)$$
$$-\int_{S}\int_{S}\int_{S}\int_{S}f(xsrk\sigma(t))d\mu(k)d\mu(s)d\mu(r)d\mu(t)]$$
$$+\int_{S}\int_{S}\int_{S}\int_{S}f(xsr\sigma(k)\sigma(t))d\mu(k)d\mu(s)d\mu(r)d\mu(t)$$
$$-\int_{S}\int_{S}f(x\sigma(t)s)d\mu(s)d\mu(t)\int_{S}f(t)d\mu(t)$$
$$+\int_{S}\int_{S}f(x\sigma(t)s)d\mu(s)d\mu(t)\int_{S}f(t)d\mu(t)-f(x)(\int_{S}f(t)d\mu(t))^{2}$$
$$-[\int_{S}\int_{S}\int_{S}f(xsrk\sigma(t))d\mu(k)d\mu(s)d\mu(r)d\mu(t)-\int_{S}f(xst)d\mu(s)d\mu(t)\int_{S}f(s)d\mu(s)]$$
$$-[\int_{S}f(xst)d\mu(s)d\mu(t)\int_{S}f(s)d\mu(s)-f(x)(\int_{S}f(t)d\mu(t))^{2}].$$
From inequalities (\ref{ahm1}), (\ref{ahm2}), (\ref{ahm3}) and the
above relations we get
$$2\int_{S}\int_{S}f(ks)d\mu(k)d\mu(s)
\int_{S}f(xs)d\mu(s)-2\int_{S}f(xst)d\mu(s)d\mu(t)\int_{S}f(s)d\mu(s)$$
$$\leq \delta \|\mu\|^{3}+\delta \|\mu\|^{3}+\frac{\delta}{2} \|\mu\|^{3}+\frac{\delta}{2} \|\mu\|\int_{S}f(s)d\mu(s)|$$
$$+\frac{\delta}{2} \|\mu\|^{2}+\frac{3\delta}{2}
\|\mu\||\int_{S}f(s)d\mu(s)|.$$Which implies that
$$|\int_{S}g(xt)d\mu(t)-g(x)\int_{S}g(t)d\mu(t)|\leq
 \frac{(5/4)\delta\|\mu\|^{3}+(1/4)\delta\|\mu\|^{2}}{(|\int_{S}f(s)d\mu(s)|)^{2}}+\frac{\delta\|\mu\|}{|\int_{S}f(s)d\mu(s)|}$$ and this
 proves (\ref{ahm7}). Now, since $g$ is unbounded and satisfies the
 inequality (\ref{ahm6}) so, from \cite{BELAID}, we deduce that $g$
 satisfies the d'Alembert's functional equation (\ref{eq6000}). We will
 show that $\int_{S}g(xt)d\mu(t)=g(x)\int_{S}g(t)d\mu(t)$ for all $x\in S.$
  $$2|g(y)||\int_{S}g(xt)d\mu(t)-g(x)\int_{S}g(t)d\mu(t)|=|\int_{S}2g(y)g(xt)d\mu(t)-2g(x)g(y)\int_{S}g(t)d\mu(t)|$$
 $$=|\int_{S}[g(xyt)+g(x\sigma(y)t)]d\mu(t)-\int_{S}g(t)d\mu(t)[g(xy)+g(x\sigma(y))]$$
 $$=|\int_{S}[g(xyt)d\mu(t)-\int_{S}g(t)d\mu(t)g(xy)+\int_{S}[g(x\sigma(y)t)d\mu(t)-\int_{S}g(t)d\mu(t)g(x\sigma(y))|$$
 $$\leq |\int_{S}[g(xyt)d\mu(t)-g(xy)\int_{S}g(t)d\mu(t)|+|\int_{S}[g(x\sigma(y)t)d\mu(t)-g(x\sigma(y))\int_{S}g(t)d\mu(t)|$$In
 view of inequality (\ref{ahm7}) we obtain
$$2|g(y)||\int_{S}g(xt)d\mu(t)-g(x)\int_{S}g(t)d\mu(t)|$$$$\leq
 2[\frac{(5/4)\delta\|\mu\|^{3}+(1/4)\delta\|\mu\|^{2}}{(|\int_{S}f(s)d\mu(s)|)^{2}}+\frac{\delta\|\mu\|}{|\int_{S}f(s)d\mu(s)|}]$$
 Since $g$ is an unbounded function on $S$ then we get
 $\int_{S}g(xt)d\mu(t)=g(x)\int_{S}g(t)d\mu(t)$ for all $x\in S.$
This completes the proof.
\end{proof} Now, we are ready to prove the main result
of the present section. We notice here that same  result has been
obtained  in \cite{BELAID} with other assumptions on $\mu$.
\begin{thm} Let $\sigma$ be an involutive morphism   of $S$. Let $\mu$ be a complex
measure that is a linear combination of Dirac measures
$(\delta_{z_{i}})_{i\in I}$, such that  $z_{i}$ is contained in the
center of $S$ for all $i\in I$.  Let $\delta>0$ be fixed.  If $f:
S\longrightarrow \mathbb{C}$ satisfies the inequality
\begin{equation}\label{ahmed}
 |\int_{S}f(xyt)d\mu(t)+\int_{S}f(x\sigma(y)t)d\mu(t)-2f(x)f(y)|\leq \delta
\end{equation} for all $x,y\in S$.
Then either $f$ is bounded and
$|f(x)|\leq\frac{\|\mu\|+\sqrt{\|\mu\|^{2}+2\delta}}{2}$  for all
$x\in S$ or $f$ is a solution of the integral Kannappan's functional
equation (\ref{eq02000}).
\end{thm}
\begin{proof} Assume that $f$ is an unbounded solution of (\ref{ahmed}).  Replacing $y$ by  $s$ in (\ref{ahmed})
and integrating the result obtained with respect to $s$ we get
\begin{equation}\label{ahmed-elqo2}
 |\int_{S}\int_{S}f(xyst)d\mu(s)d\mu(t)+\int_{S}\int_{S}f(x\sigma(y)\sigma(s)t)d\mu(s)d\mu(t)-2f(x)\int_{S}f(ys)d\mu(s)|\leq
 \delta \|\mu\|
\end{equation} for all $y\in S$. From (3.3), (3.4) and the triangle
inequality we get \begin{equation}\label{ahmed-elqo3}
 |\int_{S}f(s)d\mu(s)f(xy)+\int_{S}f(s)d\mu(s)f(x\sigma(y))-2f(x)\int_{S}f(ys)d\mu(s)|\leq
 3\delta \|\mu\|
\end{equation} for all $x,y\in S$. Since from (3.5) we have $\int_{S}f(s)d\mu(s)\neq0$. Then the inequality (\ref {ahmed-elqo3}) can be written as follows
\begin{equation}\label{ahmed-elqo4}
 |f(xy)+f(x\sigma(y))-2f(x)g(y)|\leq
 \frac{3\delta \|\mu\|}{|\int_{S}f(s)d\mu(s)|}
\end{equation} for all $x,y\in S$ and where $g$ is the function
defined in Lemma 3.1. Now, by using same computation used in [6,
Theorem 2.2(iii)]  we conclude that $f,g$ are solutions of Wilson's
functional equation
\begin{equation}\label{wilson}
f(xy)+f(x\sigma(y))=2f(x)g(y)\end{equation}for all $x,y \in S.$ By
replacing $x$ by $t$ in (\ref{wilson}) and integrating the result
obtained with respect to $t$ we get
$\int_{S}f(ty)d\mu(t)+\int_{S}f(t\sigma(y))d\mu(t)=2g(y)\int_{S}f(t)d\mu(t)$.
Since $f\circ\sigma=f$ then we get
$$\int_{S}f(ty)d\mu(t)+\int_{S}f(t\sigma(y))d\mu(t)$$$$=\int_{S}f(yt)d\mu(t)+\int_{S}f(y\sigma(t))d\mu(t)=2f(y)\int_{S}g(t)d\mu(t).$$
Then we have $f(y)\int_{S}g(t)d\mu(t)=g(y)\int_{S}f(t)d\mu(t)$. So,
$g=\frac{\int_{S}g(t)d\mu(t)}{\int_{S}f(t)d\mu(t)}f$.\\  For all
$x,y\in S$ we have
\begin{equation}\label{akk1}
\int_{S}f(xyt)d\mu(t)+\int_{S}f(x\sigma(y)t)d\mu(t)\end{equation}
$$=\int_{S}[f(xty)+f(xt\sigma(y)]d\mu(t)=2\int_{S}f(xt)d\mu(t)g(y)=2\beta
f(x)f(y),$$ where
$\beta=\frac{(\int_{S}g(t)d\mu(t))^{2}}{\int_{S}f(t)d\mu(t)}$.
Substituting this into (3.1) we obtain \\$|2(\beta
-1)f(y)f(x)|\leq\delta$ for all $x,y\in S.$ Since $f$ is assumed to
be unbounded then we deduce that $\beta=1$  and then from
(\ref{akk1}) we deduce that $f$ is a solution of (\ref{eq02000}).
This completes the proof.\end{proof} \subsection{Superstability of
the integral Kannappan's functional equation (\ref{eq02000}) on
Monoids} If $S$ is a monoid (A semigroup with identity element $e$)
then by elementary computations we verify that the superstability of
the integral Kannappan's functional equation follows from the
superstability of d'Alembert's functional equation
(1.7).\begin{prop} Let $M$ be a topological monoid. Let $\sigma$ be
an involutive anti-automorphism of $M$ and let $\mu$ a complex
measure with compact support. Let $\delta>0$ be fixed.  If a
continuous function $f: M\longrightarrow \mathbb{C}$ satisfies the
inequality
\begin{equation}\label{man1}
 |\int_{M}f(xyt)d\mu(t)+\int_{M}f(x\sigma(y)t)d\mu(t)-2f(x)f(y)|\leq \delta
\end{equation} for all $x,y\in M$.
Then either $f$ is bounded  or $f$ is a solution of the integral
Kannappan's functional equation (\ref{eq02000}).\end{prop}
\begin{proof}Let $f$ be an unbouded contious function which satisfies (\ref{man1}). Taking $y=e$ in (\ref{man1}) we get
\begin{equation}\label{man2}
 |\int_{M}f(xt)d\mu(t)-f(e)f(x)|\leq \frac{\delta}{2}
\end{equation} for all $x\in M$. Since $f$ is unbounded then
$f(e)\neq0$, because if $f(e)=0$ the functions $(x,y)\longmapsto
\int_{M}f(xyt)d\mu(t)$; $(x,y)\longmapsto
\int_{M}f(x\sigma(y)t)d\mu(t)$ are bounded and from (\ref{man1}) and
the triangle inequality we get $f$ a bounded function on $S$. This
contradict the assumption that $f$ is unbounded. Now, From
(\ref{man1}), (\ref{man2}) and the triangle inequality we obtain
\begin{equation}\label{man3}
|f(e)f(xy)+f(e)f(x\sigma(y))-2f(x)f(y)|\end{equation}
$$\leq|f(e)f(xy)-\int_{M}f(xyt)d\mu(t)|+|f(e)f(x\sigma(y))-\int_{M}f(x\sigma(y)t)d\mu(t)|$$
$$+|\int_{M}f(xyt)d\mu(t)+\int_{M}f(x\sigma(y)t)d\mu(t)-2f(x)f(y)|\leq\frac{\delta}{2}+\frac{\delta}{2}+\delta=2\delta.$$
Inequality which can be written as follows
\begin{equation}\label{man4}
|f(xy)+f(x\sigma(y))-2f(x)\frac{f(y)}{f(e)}|\leq\frac{2\delta}{|f(e)|},\;
x,y\in M.\end{equation}From [6, Theorem 2.2(iii)]  we deduce  that
$f,g$ are solutions of Wilson's functional equation
\begin{equation}\label{man4}
f(xy)+f(x\sigma(y))=2f(x)\frac{f(y)}{f(e)}
\end{equation} for all $x,y\in M$, then from \cite{07} $f$ is central.
So,
$$\int_{M}f(xyt)d\mu(t)+\int_{M}f(x\sigma(y)t)d\mu(t)=\int_{M}f(txy)d\mu(t)+\int_{M}f(tx\sigma(y))d\mu(t)$$
$$=\int_{M}[f(tx(y))+f((tx)\sigma(y))]d\mu(t)=2\frac{f(y)}{f(e)}\int_{M}f(tx)d\mu(t)=2\frac{f(y)}{f(e)}\int_{M}f(xt)d\mu(t).$$
Substituting this into (\ref{man1}) after computation we get
$|f(y)(f(x)-\frac{\int_{M}f(xt)d\mu(t)}{f(e)})|\leq\frac{\delta}{2}$
for all $x,y\in M.$ Since $f$ is unbounded then
$f(x)=\frac{\int_{M}f(xt)d\mu(t)}{f(e)}$ for all $x\in S.$ Thus, for
all $x,y\in M$ we get
$$\int_{M}f(xyt)d\mu(t)+\int_{M}f(x\sigma(y)t)d\mu(t)=2\frac{f(y)}{f(e)}\int_{M}f(xt)d\mu(t)=2f(x)f(y).$$  That is $f$ satisfies the integral
Kannappan's functional equation (\ref{eq02000}). This completes the
proof.
\end{proof}
\section{Solutions of the
functional equation (\ref{eq030000})} The solutions of the
functional equation (\ref{eq030000}) with $\sigma$ an involutive
anti-automorphism  are explicitly obtained  by Elqorachi
\cite{Elqorachi} on semigroups not necessarily abelian in terms of
multiplicative functions. In this section we express the solutions
of (\ref{eq030000}) where $\sigma$ is an involutive automorphism  in
terms of multiplicative functions. The following lemma is obtained
in \cite{Elqorachi} for the case where $\sigma$ is an involutive
anti-automorphism. It still holds for the case where $\sigma$ is an
involutive automorphism.\begin{lem} Let $\sigma$: $S\longrightarrow
S$ be a morphism of $S$ . Let $\mu$ be a complex measure that is a
linear combination of Dirac measures $(\delta_{z_{i}})_{i\in I}$,
such that $z_{i}$ is contained in the center of $S$ for all $i\in
I$. Let $f$ be a non-zero solution of equation (\ref{eq030000}).
Then for all $x\in S$ we have
\begin{equation}\label{eq-ah77}
    f(x)=-f(\sigma(x)),
\end{equation}
\begin{equation}\label{eq-ah88}
    \int_{S}f(t)d\mu(t)\neq 0,
\end{equation}
\begin{equation}\label{eq-ah99'}
    \int_{S}\int_{S}f(ts)d\mu(t)d\mu(s)=\int_{S}\int_{S}f(\sigma(t)s)d\mu(t)d\mu(s)=0,
\end{equation}
\begin{equation}\label{eq-ah99}
    \int_{S}\int_{S}f(x\sigma(t)s)d\mu(t)d\mu(s)=f(x)\int_{S}f(t)d\mu(t),
\end{equation}
\begin{equation}\label{eq-ah100}
    \int_{S}\int_{S}f(xts)d\mu(t)d\mu(s)=-f(x)\int_{S}f(t)d\mu(t),
\end{equation}
\begin{equation}\label{eq-ah111}
   \int_{S}f(\sigma(x)t)d\mu(t)=\int_{S}f(xt)d\mu(t).
\end{equation}
The function defined by
$$g(x)\;:=\frac{\int_{S}f(xt)d\mu(t)}{\int_{S}f(s)d\mu(s)}\;for \;x\in S$$ is a non-zero
solution of d'Alembert's functional equation (\ref{eq6000}).
Furthermore,
$$\int_{S}\int_{S}g(ts)d\mu(t)d\mu(s)\neq
0;\;\;\int_{S}g(s)d\mu(s)=0.$$ That is $Jf=g\in \mathcal{B}$, where
$J$ and $\mathcal{B}$ are the function and the set defined in
Theorem 4.2.
 \end{lem}
 \begin{thm} Let $\sigma$: $S\longrightarrow S$ be an involutive
morphism of $S$. Let $\mu$ be a complex measure that is a linear
combination of Dirac measures $(\delta_{z_{i}})_{i\in I}$, such that
  $z_{i}$ is contained in the center of $S$ for all $i\in I$.  Let
$\mathcal{B}$ consists of the solution  $g:$ $S\longrightarrow
\mathbb{C}$ of d'Alembert's functional equation (\ref{eq6000}) such
 that $\int_{S}g(t)d\mu(t)=0$ and
$\int_{S}\int_{S}g(st)d\mu(s)d\mu(t)\neq0$.  Let $\mathcal{V}$
consist of the non-zero solutions of the extension of Van Vleck's
functional equation (\ref{eq030000}). Then
 the function $J$: $\mathcal{V}\longrightarrow \mathcal{B} $
 defined by
\begin{equation}\label{V11}
Jf(x)=\frac{\int_{S}f(xt)d\mu(t)}{\int_{S}f(t)d\mu(t)},\; x\in
S\end{equation} is a bijection of $\mathcal{V}$ onto $\mathcal{B}$.
In particular $ J(\mathcal{V})=\mathcal{B}.$
\end{thm}\begin{proof} From  Lemma 4.1 the formula (\ref{V11}) makes sense, and  we have $g := Jf \in \mathcal{B}$ for any $f\in \mathcal{V}.$\\
-Injection: Let $f_1$ and $f_2$ be two non-zero solutions of
(\ref{eq030000}). If $Jf_1=Jf_2$ then we get
 \begin{equation}\label{VAN}
 \int_{S}f_2(t)d\mu(t)\int_{S}f_1(xt)d\mu(t)
 =\int_{S}f_1(t)d\mu(t)\int_{S}f_2(xt)d\mu(t)\end{equation} for all
 $x\in S.$ Since
 $f_1$ and $f_2$ are solutions of (\ref{eq030000}) then we have
 \begin{equation}\label{VAn1}
 \int_{S}f_1(x\sigma(y)t)d\mu(t)-\int_{S}f_1(xyt)d\mu(t)=2f_1(x)f_1(y)\end{equation}
and
 \begin{equation}\label{VAn2}
 \int_{S}f_2(x\sigma(y)t)d\mu(t)-\int_{S}f_2(xyt)d\mu(t)=2f_2(x)f_2(y).\end{equation}
 By multiplying (\ref{VAn1}) by $\int_{S}f_2(t)d\mu(t)$ and using (\ref{VAN}) we
 get
 \begin{equation}\label{ELQ}
 2f_1(x)f_1(y)\int_{S}f_2(t)d\mu(t)=2f_2(x)f_2(y)\int_{S}f_1(t)d\mu(t).\end{equation}By
 replacing $y$ by $s$ in (\ref{ELQ}) and integrating the result
 obtained with respect to $s$ we get
 $2f_1(x)\int_{S}f_1(s)d\mu(s)\int_{S}f_2(t)d\mu(t)=2f_2(x)\int_{S}f_2(s)d\mu(s)\int_{S}f_1(t)d\mu(t)$.
 Since $\int_{S}f_2(s)d\mu(s)\int_{S}f_1(t)d\mu(t)\neq 0$, then
 $f_1=f_2$.
 \\Surjection: Let $g\in \mathcal{B}$. First we notice that since $g$ is a solution of (\ref{eq6000}) and $\int_{S}g(s)d\mu(s)=0$ then if we let $y=s$ in  (\ref{eq6000}) and integrating
  the result
 obtained with respect to $s$ we deduce that  $\int_{S}g(x\sigma(s))d\mu(s)=-\int_{S}g(xs)d\mu(s)$.  We may define
$f:$ $S\longrightarrow \mathbb{C}$ by
$$f(x)=\frac{1}{2}(\int_{S}g(x\sigma(s))d\mu(s)-\int_{S}g(xs)d\mu(s))=\int_{S}g(x\sigma(s))d\mu(s)=-\int_{S}g(xs)d\mu(s).$$For
all $x,y\in S$ we have $$
\int_{S}f(x\sigma(y)t)d\mu(t)-\int_{S}f(xyt)d\mu(t)=\int_{S}\int_{S}g(x\sigma(y)t\sigma(s))d\mu(t)d\mu(s)-\int_{S}\int_{S}g(xyt\sigma(s))d\mu(t)d\mu(s)$$
$$=\int_{S}\int_{S}g(xt\sigma(ys))d\mu(t)d\mu(s)+\int_{S}\int_{S}g(xtys)d\mu(t)d\mu(s)=2\int_{S}g(xt)d\mu(t)\int_{S}g(ys)d\mu(s)=2f(x)f(y).$$
Furthermore,
$$\int_{S}f(s)d\mu(s)=\int_{S}\int_{S}g(s\sigma(t))d\mu(s)d\mu(t)=-\int_{S}\int_{S}g(st)d\mu(s)d\mu(t)\neq0.$$
Thus, we get $f\neq 0$. On the other hand for all $x\in S$ we have
$$Jf(x)=\frac{\int_{S}f(xt)d\mu(t)}{\int_{S}f(t)d\mu(t)}$$
$$=\frac{\int_{S}\int_{S}g(xt\sigma(s))d\mu(t)d\mu(s)}{\int_{S}\int_{S}g(t\sigma(s))d\mu(t)d\mu(s)}$$
$$\frac{\int_{S}\int_{S}g(xt\sigma(s))d\mu(t)d\mu(s)+\int_{S}\int_{S}g(xt\sigma(s))d\mu(t)d\mu(s)}{2\int_{S}\int_{S}g(t\sigma(s))d\mu(t)d\mu(s)}$$
$$=\frac{2g(x)\int_{S}\int_{S}g(t\sigma(s))d\mu(t)d\mu(s)}{2\int_{S}\int_{S}g(t\sigma(s))d\mu(t)d\mu(s)}=g(x).$$
This completes the proof. \end{proof} In \cite{Elqorachi} we use
[33, Proposition 8.14] to derive  the form of the solutions of
(\ref{eq030000}) where $\sigma$ is an involutive anti-automorphism
of $S$. This reasoning no longer works for the present situation. We
will use an elementary  approach which  works for both situations.
\begin{thm} Let $\sigma$: $S\longrightarrow S$ be   a morphism
 of $S$. Let $\mu$ be a complex-measure  that is a linear combination of Dirac measures
$(\delta_{z_{i}})_{i\in I}$, such that   $z_{i}$ is contained in the
center of $S$ for all $i\in I$. The non-zero central solutions of
the integral Van Vleck's functional equation (\ref{eq030000})  are
the functions of the form
$$f=\frac{\chi- \chi\circ\sigma}{2}\int_{S}\chi(\sigma(t))d\mu(t),
$$ where $\chi$ : $S\longrightarrow \mathbb{C}$ is a multiplicative
function such that $\int_{S}\chi(t)d\mu(t)\neq 0$ and
$\int_{S}\chi(\sigma(t)d\mu(t))=-\int_{S}\chi(t)d\mu(t)$.
\\Furthermore, if $S$ is a topological semigroup and that $\sigma$ :
$S\longrightarrow S$ is continuous, then the non-zero solution $f$
of equation (\ref{eq030000}) is continuous, if and only if $\chi$ is
continuous.
\end{thm}\begin{proof} Let $f$ be a non-zero solution of
(\ref{eq030000}). Replacing $x$ by $xs$ in (\ref{eq030000}) and
integrating the result obtained with respect to $s$ we get
\begin{equation}\label{haj1}
\int_{S}\int_{S}f(x\sigma(y)st)d\mu(s)d\mu(t)-\int_{S}\int_{S}f(xyst)d\mu(s)d\mu(t)=2f(y)\int_{S}f(xs)d\mu(s).\end{equation}
By using  (\ref{eq-ah100}) equation (\ref{haj1}) can be written as
follows
\begin{equation}\label{haj2}
-f(x\sigma(y))+f(xy)=2f(y)g(x),\;x,y\in S,\end{equation}  where $g$
is the function defined in Lemma 4.1. If we replace $y$ by $ys$ in
(\ref{eq030000}) and integrate the result obtained with respect to
$s$ we get
\begin{equation}\label{haj3}
\int_{S}\int_{S}f(x\sigma(y)\sigma(s)t)d\mu(s)d\mu(t)-\int_{S}\int_{S}f(xyst)d\mu(s)d\mu(t)=2f(x)\int_{S}f(ys)d\mu(s).\end{equation}
By using  (\ref{eq-ah100}) we obtain that
\begin{equation}\label{haj4}
f(x\sigma(y))+f(xy)=2f(x)g(y),\;x,y\in S\end{equation}. By adding
(\ref{haj4}) and (\ref{haj2}) we get that the pair $f,g$ satisfies
the sine addition law
$$f(xy)=f(x)g(y)+f(y)g(x)\; \text{for all }\; x,y\in S.$$ Now, in
view of [12, Lemma 3.4], [33, Theorem 4.1]    $g$ is abelian. Since
$g$ is a non-zero solution of d'Alembert's functional equation (1.7)
there exists a non-zero multiplicative function $\chi$:
$S\longrightarrow \mathbb{C}$ such that
$g=\frac{\chi+\chi\circ\sigma}{2}$. The rest of the proof is similar
to the one used in  \cite{Elqorachi}. This completes the
proof.\end{proof}\begin{cor}Let $S$ be a semigroup, let $\sigma$ be
an involutive automorphism of $S$. Let $g$ be a solution of
d'Alembert's functional equation (\ref{eq6000}). If there exists a
complex measure $\mu$ that is a linear combination of Dirac measures
$(\delta_{z_{i}})_{i\in I}$, such that  $z_{i}$ is contained in the
center of $S$ for all $i\in I$ and $\int_{S}g(t)d\mu(t)=0$;
$\int_{S}\int_{S}g(ts)d\mu(t)d\mu(t)\neq0$. Then there exists a
non-zero multiplicative function $\chi$: $S\longrightarrow
\mathbb{C}$ such that
$g=\frac{\chi+\chi\circ\sigma}{2}$.\end{cor}\begin{proof} Let $g:$
$S\longrightarrow \mathbb{C}$ be a non-zero function which satisfies
the conditions of Corollary 4.4.  From Theorem 4.2 there exists a
non-zero solution of the integral Van Vleck's functional equation
(\ref{eq030000}) such that $Tf=g$. From the proof of Theorem 4.3, we
get that $g$ is an abelian solution of d'Alembert's functional
equation (\ref{eq6000}). That is there exists a non-zero
multiplicative function $\chi$: $S\longrightarrow \mathbb{C}$ such
that $g=\frac{\chi+\chi\circ\sigma}{2}$. This completes the proof.
\end{proof}
\section{The superstability of the integral Van Vleck's functional equation (\ref{eq030000})} In the present section we
prove the superstability theorem  of the integral Van Vleck's
functional equation (\ref{eq030000}) on semigroups. First, we prove
he following useful lemma.
\begin{lem} Let  et $\sigma$ be an involutive morphism  of $S$. Let $\mu$ be a complex
measure that is a linear combination of Dirac measures
$(\delta_{z_{i}})_{i\in I}$, such that  $z_{i}$ is contained in the
center of $S$ for all $i\in I$. Let $\delta>0$ be fixed.  If $f:
S\longrightarrow \mathbb{C}$ is an unbounded function which
satisfies the inequality
\begin{equation}\label{elq1}
 |\int_{S}f(x\sigma(y)t)d\mu(t)-\int_{S}f(xyt)d\mu(t)-2f(x)f(y)|\leq \delta
\end{equation} for all $x,y\in S$. Then, for all $x\in S$
\begin{equation}\label{elq2}
    f(\sigma(x))=-f(x),
\end{equation}
\begin{equation}\label{elq3}
|\int_{S}\int_{S}f(x\sigma(s)t)d\mu(s)d\mu(t)-f(x)\int_{S}f(t)d\mu(t)|\leq
\frac{\delta}{2}\|\mu\|,
\end{equation}
\begin{equation}\label{elq4}
|\int_{S}\int_{S}f(xst)d\mu(s)d\mu(t)+f(x)\int_{S}f(t)d\mu(t)|\leq
\frac{3\delta\|\mu\|}{2},
\end{equation}
\begin{equation}\label{elq5}
\int_{S}f(t)d\mu(t)\neq 0,
\end{equation}
\begin{equation}\label{elq6000}
 \int_{S}\int_{S}f(st)d\mu(s)d\mu(t)=0,\end{equation}
 \begin{equation}\label{elq7000}
 |\int_{S}f(xs)d\mu(s)-\int_{S}f(\sigma(x)s)d\mu(s)|\leq
 \frac{4\delta\|\mu\|^{2}}{|\int_{S}f(s)d\mu(s)|}.
\end{equation}
The function $g$ defined by
\begin{equation}\label{elq8}g(x)=
\frac{\int_{S}f(xt)d\mu(t)}{\int_{S}f(t)d\mu(t)}\;\text{for}\;x\in
S\end{equation} is unbounded on $S$ and satisfies the following
inequality
\begin{equation}\label{elq9}
 |g(xy)+g(x\sigma(y))-2g(x)g(y)|\leq
 \frac{3\delta\|\mu\|^{2}}{(|\int_{S}f(s)d\mu(s))^{2}|}\;\text{for all}\;x,y\in S.
\end{equation} Furthermore, \\1) $\int_{S}g(t)d\mu(t)=0$; $\int_{S}\int_{S}g(ts)d\mu(t)d\mu(t)\neq0$ \\2)   $g$ is an abelian solution of
d'Alembert's functional equation (\ref{eq6000}) and $Jf=g\in
\mathcal{B}$.\\3)  $f,g$ are solutions of Wilson's functional
equation
\begin{equation}\label{wilson}
f(xy)+f(x\sigma(y))=2f(x)g(y)\end{equation}for all $x,y \in S.$
\end{lem}
\begin{proof}Equation (5.2): If we replace $y$ by $\sigma(y)$ in (\ref{elq1}) we get
\begin{equation}\label{elq11}
 |\int_{S}f(xyt)d\mu(t)-\int_{S}f(x\sigma(y)t)d\mu(t)-2f(x)f(\sigma(y))|\leq \delta
\end{equation} for all $x,y\in S$. By adding the result of (\ref{elq1})
and (\ref{elq11}) and using the triangle inequality we obtain
$|2f(x)(f(y)+f(\sigma(y)))|\leq2\delta$ for all $x\in S$. Since $f$
is assumed to be unbounded then we get (\ref{elq2}).\\Equation
(5.3):  By replacing $x$ by $\sigma(s)$ in (\ref{elq1}), using
(\ref{elq2}) and integrating the result obtained with respect to $s$
we have
\begin{equation}\label{elq12}
 |-\int_{S}\int_{S}f(y\sigma(s)t)d\mu(s)d\mu(t)-\int_{S}\int_{S}f(y\sigma(s)t)d\mu(s)d\mu(t)+2f(y)\int_{S}f(s)d\mu(s)|\leq \delta\|\mu\|
\end{equation} for all $y\in S.$ This proves  (\ref{elq3}).\\Equation (5.4):
Taking $y=s$ in  (\ref{elq1}) and integrating the result obtained
with respect to $s$ we get
\begin{equation}\label{elq13}
 |\int_{S}\int_{S}f(x\sigma(s)t)d\mu(s)d\mu(t)-\int_{S}\int_{S}f(xst)d\mu(s)d\mu(t)-2f(x)\int_{S}f(s)d\mu(s)|\leq \delta\|\mu\|
\end{equation} for all $x\in S.$
Since
$$|\int_{S}\int_{S}f(xst)d\mu(s)d\mu(t)+f(x)\int_{S}f(s)d\mu(s)|$$
$$=|\int_{S}\int_{S}f(xst)d\mu(s)d\mu(t)+2f(x)\int_{S}f(s)d\mu(s)-\int_{S}\int_{S}f(x\sigma(s)t)d\mu(s)d\mu(t)
$$$$+\int_{S}\int_{S}f(x\sigma(s)t)d\mu(s)d\mu(t)-f(x)\int_{S}f(s)d\mu(s)|.$$
So, by using (\ref{elq3}), (5.12) and the triangle inequality we
deduce (\ref{elq4}).\\Equation (5.5): $f$ is assumed to be an
unbounded solution of the inequality (\ref{elq1}) then $f\neq0$. Now
assume that $\int_{S}f(t)d\mu(t)=0$. By replacing $x$ by $xs$ in
(\ref{elq1}) and integrating the result obtained with respect to $s$
we get
\begin{equation}\label{elq14}
 |\int_{S}\int_{S}f(x\sigma(y)st)d\mu(s)d\mu(t)-\int_{S}\int_{S}f(xyst)d\mu(s)d\mu(t)-2f(y)\int_{S}f(xs)d\mu(s)|\leq \delta\|\mu\|
\end{equation} for all $x,y\in S.$ Since $$2f(y)\int_{S}f(xs)d\mu(s)=2f(y)\int_{S}f(xs)d\mu(s)+\int_{S}\int_{S}f(xyst)d\mu(s)d\mu(t)
-\int_{S}\int_{S}f(x\sigma(y)st)d\mu(s)d\mu(t)$$
$$-(\int_{S}\int_{S}f(xyst)d\mu(s)d\mu(t)+f(xy)\int_{S}f(s)d\mu(s))$$$$+\int_{S}\int_{S}f(x\sigma(y)st)d\mu(s)d\mu(t)
+f(x\sigma(y))\int_{S}f(s)d\mu(s)$$$$+f(xy)\int_{S}f(s)d\mu(s)-f(x\sigma(y))\int_{S}f(s)d\mu(s).$$So,
by using (\ref{elq4}), (\ref{elq1}), $\int_{S}f(t)d\mu(t)=0$ and the
triangle inequality we get $y\longmapsto f(y)\int_{S}f(xs)d\mu(s)$
is a bounded function on $S$, since $f$ is unbounded then we obtain
$\int_{S}f(xs)d\mu(s)=0$ for all $x \in S$. By substituting this
into (\ref{elq1}) we get $f$ a bounded function on $S$ and this
contradict the assumption that $f$ is an unbounded function. So, we
have (\ref{elq5}).\\Equation (5.9):  By using similar computation
used above the function $g$ defined by (\ref{elq8}) is an unbounded
function on $S$. Furthermore,
\begin{equation}\label{elq15}
\int_{S}f(s)d\mu(s)\int_{S}f(k)d\mu(k)[g(xy)+g(x\sigma(y))-2g(x)g(y)]
 \end{equation}
 $$=\int_{S}f(k)d\mu(k)\int_{S}f(xyt)d\mu(t)+\int_{S}f(s)d\mu(s)\int_{S}f(x\sigma(y)t)d\mu(t)-2\int_{S}f(xs)d\mu(s)\int_{S}f(ys)d\mu(s)$$
 $$=\int_{S}[f(xyt)\int_{S}f(k)d\mu(k)+\int_{S}\int_{S}f(xytks)d\mu(k)d\mu(s)]d\mu(t)$$
 $$+\int_{S}[f(x\sigma(y)t)\int_{S}f(s)d\mu(s)-\int_{S}\int_{S}f(x\sigma(y)t\sigma(k)s)d\mu(k)d\mu(s)]d\mu(t)$$
 $$+\int_{S}\int_{S}[\int_{S}f(xs\sigma(yk)t)d\mu(t)-\int_{S}f(xsykst)d\mu(t)-2f(xs)f(yk)]d\mu(k)d\mu(s).$$ So, by using (\ref{elq3}),
 (\ref{elq4}) and (\ref{elq1}) we get
 $$|\int_{S}f(s)d\mu(s)\int_{S}f(k)d\mu(k)[g(xy)+g(x\sigma(y))-2g(x)g(y)]|$$
 $$\leq\frac{3\delta\|\mu\|^{2}}{2}+\frac{\delta\|\mu\|^{2}}{2}+\delta\|\mu\|^{2}
 =3\delta\|\mu\|^{2}.$$ Which proves (\ref{elq9}).
\\Equation (5.6): Since
 $g$ is unbounded so, from \cite{BELAID} $g$ satisfies the d'Alembert's functional
 equation (\ref{eq6000}). From (\ref{elq3}), (\ref{elq4}) and the triangle inequality we have
\begin{equation}\label{elq2000}
|\int_{S}\int_{S}f(x\sigma(s)t)d\mu(s)d\mu(t)+\int_{S}f(xst)d\mu(s)d\mu(t)|\leq
2\delta\|\mu\|,
\end{equation} for all $x,y\in S.$ Since
$g=\frac{\int_{S}f(xk)d\mu(k)}{\int_{S}f(t)d\mu(t)}$ So, the
inequality (\ref{elq2000}) can be written as follows
$$|\int_{S}f(k)d\mu(k)\int_{S}g(x\sigma(k))d\mu(k)+\int_{S}f(k)d\mu(k)\int_{S}g(xk)d\mu(k)|\leq2\delta\|\mu\|.$$
On the other hand $g$ is a solution of d'Alembert's functional
equation (1.9) then we get
$|2g(x)\int_{S}g(k)d\mu(k)|\leq\frac{2\delta\|\mu\|}{|\int_{S}f(k)d\mu(k)|}$
for all $x\in S.$ Since $g$ is unbounded then we deduce that
$\int_{S}g(k)d\mu(k)=0$. That is
$\int_{S}\int_{S}f(st)d\mu(s)d\mu(t)=0$. Which proves
(\ref{elq6000}). \\Equation (5.7): By replacing $x$ by $sk$ in
(\ref{elq1}), integrating the result with respect to $s$ and $k$ and
using (\ref{elq6000}) we obtain
\begin{equation}\label{elq2001}
 |\int_{S}\int_{S}f(\sigma(y)skt)d\mu(s)d\mu(k)d\mu(t)-\int_{S}\int_{S}f(yskt)d\mu(s)d\mu(k)d\mu(t)|\leq
 \delta\|\mu\|^{2}
\end{equation} for all $y\in S.$ Since
$$\int_{S}\int_{S}\int_{S}f(\sigma(y)skt)d\mu(s)d\mu(k)d\mu(t)-\int_{S}\int_{S}\int_{S}f(yskt)d\mu(s)d\mu(k)d\mu(t)$$
$$=\int_{S}\int_{S}\int_{S}f(\sigma(y)skt)d\mu(s)d\mu(k)d\mu(t)+\int_{S}f(\sigma(y)t)d\mu(t)\int_{S}f(s)d\mu(s)$$
$$-(\int_{S}\int_{S}\int_{S}f(yskt)d\mu(s)d\mu(k)d\mu(t)+\int_{S}f(yt)d\mu(t)\int_{S}f(s)d\mu(s))$$
$$-(\int_{S}f(\sigma(y)t)d\mu(t)-\int_{S}f(yt)d\mu(t))\int_{S}f(s)d\mu(s).$$
So from  (\ref{elq2001}), (\ref{elq4}) and the triangle inequality
we get
$$|\int_{S}f(\sigma(y)t)d\mu(t)-\int_{S}f(yt)d\mu(t)|\leq\frac{4\delta\|\mu\|^{2}}{|\int_{S}f(s)d\mu(s)|}.$$
This proves (\ref{elq7000}).\\ Equation (\ref{wilson}) From (5.1),
(5.3), (5.4) and the triangle inequality we get
\begin{equation}\label{hamo1}
 |\int_{S}f(s)d\mu(s)f(xy)+\int_{S}f(s)d\mu(s)f(x\sigma(y))-2f(x)\int_{S}f(ys)d\mu(s)|\end{equation}
$$\leq|\int_{S}f(s)d\mu(s)f(xy)+\int_{S}\int_{S}f(xyst)d\mu(s)d\mu(t)| $$
$$+|\int_{S}f(s)d\mu(s)f(x\sigma(y))-\int_{S}\int_{S}f(x\sigma(y)st)d\mu(s)d\mu(t)|$$
$$+|\int_{S}\int_{S}f(x\sigma(y)st)d\mu(s)d\mu(t)-\int_{S}\int_{S}f(xyst)d\mu(s)d\mu(t)-2f(x)\int_{S}f(ys)d\mu(s)|$$
$$\leq
 3\delta \|\mu\|$$  for all $x,y\in S$. Since from (5.5) we have $\int_{S}f(s)d\mu(s)\neq0$. Then the inequality (\ref {hamo1}) can be written as follows
\begin{equation}\label{hamo3}
 |f(xy)+f(x\sigma(y))-2f(x)g(y)|\leq
 \frac{3\delta \|\mu\|}{|\int_{S}f(s)d\mu(s)|}
\end{equation} for all $x,y\in S$ and where $g$ is the function
defined in Lemma 5.1. Now, by using same computation used in [6,
Theorem 2.2(iii)]  we conclude that $f,g$ are solutions of Wilson's
functional equation (\ref{wilson}). This completes the proof.
\end{proof} Now, we are ready to prove the main result of the
present section.
\begin{thm} Let   $\sigma$ be an involutive morphism   of $S$. Let $\mu$ be a complex
measure that is a linear combination of Dirac measures
$(\delta_{z_{i}})_{i\in I}$, such that  $z_{i}$ is contained in the
center of $S$ for all $i\in I$.  Let $\delta>0$ be fixed.  If $f:
S\longrightarrow \mathbb{C}$ satisfies the inequality
\begin{equation}\label{ahmedred}
 |\int_{S}f(x\sigma(y)t)d\mu(t)-\int_{S}f(xyt)d\mu(t)-2f(x)f(y)|\leq \delta
\end{equation} for all $x,y\in S$.
Then either $f$ is bounded and
$|f(x)|\leq\frac{\|\mu\|+\sqrt{\|\mu\|^{2}+2\delta}}{2}$  for all
$x\in S$ or $f$ is a solution of the integral Van Vleck's functional
equation (\ref{eq030000}).
\end{thm}
\begin{proof} Assume that $f$ is an unbounded solution of (\ref{ahmedred}).  From Lemma 5.1 (3) $f,g$
are solutions of Wilson's functional equation (\ref{wilson}). Taking
$y=s$ in (\ref{wilson}) and integrating the result obtained with
respect to $s$ we get
\begin{equation}\label{2016}
\int_{S}f(xs)d\mu(s)+\int_{S}f(x\sigma(s))d\mu(s)=0
\end{equation} because $\int_{S}g(s)d\mu(s)=0$. By replacing $y$ by
$s\sigma(k)$ in in (\ref{wilson}) and integrating the result
obtained with respect to $s$ and $k$ we obtain
$$
\int_{S}\int_{S}f(xs\sigma(k))d\mu(s)d\mu(k)+\int_{S}\int_{S}f(xs\sigma(k))d\mu(s)d\mu(k)=2f(x)\int_{S}\int_{S}g(s\sigma(k))d\mu(s)d\mu(k).
$$
That is \begin{equation}\label{2017}
\int_{S}\int_{S}f(xs\sigma(k))d\mu(s)d\mu(k)=f(x)\int_{S}\int_{S}g(s\sigma(k))d\mu(s)d\mu(k).
\end{equation}Now from (5.3) and (\ref{2017}) we get $$|f(x)(\int_{S}\int_{S}g(s\sigma(k))d\mu(s)d\mu(k)-
\int_{S}f(t)d\mu(t))|\leq\frac{\delta\|\mu\|}{2}$$ for all $x\in S.$
Since $f$ is assumed to be unbounded then we get
\begin{equation}\label{2018}
\int_{S}\int_{S}g(s\sigma(k))d\mu(s)d\mu(k)=\int_{S}f(t)d\mu(t).
\end{equation} The function $g$ satisfies d'Alembert's functional equation (1.7) and  $\int_{S}g(s)d\mu(s)=0$ then we have
 $\int_{S}g(yk)d\mu(s)=-\int_{S}g(y\sigma(k))d\mu(s)$ for all $y\in S$. So,  by using the definition of $g$,
equations (\ref{2017}) and (\ref{2018})  we have
\begin{equation}\label{2019}
\int_{S}g(yk)d\mu(k)=-\int_{S}g(y\sigma(k))d\mu(k)=\frac{-\int_{S}\int_{S}f(y\sigma(k)t)d\mu(k)d\mu(t)}{\int_{S}f(s)d\mu(s)}\end{equation}$$
=\frac{-f(y)\int_{S}\int_{S}g(\sigma(k)t)d\mu(k)d\mu(t)}{\int_{S}f(s)d\mu(s)}=\frac{-f(y)\int_{S}f(t)d\mu(t)}{\int_{S}f(s)d\mu(s)}=-f(y).
$$ Finally, from  (\ref{wilson}), (\ref{2016}) and (\ref{2019}) for all $x,y\in
S$ we have
$$\int_{S}f(x\sigma(y)t)d\mu(t)-\int_{S}f(xyt)d\mu(t)=-\int_{S}f(x\sigma(y)\sigma(t))d\mu(t)-\int_{S}f(xyt)d\mu(t)$$
$$=-\int_{S}[f(x\sigma(yt))d\mu(t)+f(xyt)]d\mu(t)]$$
$$=-2f(x)\int_{S}g(yt)d\mu(t)=2f(x)f(y).$$ That is $f$ is a solution
of Van Vleck's functional equation (\ref{eq030000}). This completes
the proof.

\end{proof}


\end{document}